\documentclass{article}

\usepackage{amssymb,amsmath,amsthm,enumitem}

\setlist{nolistsep}        

\newtheorem{theorem}{Theorem}
\newtheorem*{maintheorem}{Main Theorem}
\newtheorem{lemma}[theorem]{Lemma}
\newtheorem{corollary}[theorem]{Corollary}

\newcommand{\iin}{\!\in\!}
\newcommand{\norm}[1]{\lVert#1\rVert}
\newcommand{\posint}{\omega^\circ}
\newcommand{\Call}{\mathsf{C}}
\newcommand{\Ub}{\mathsf{U_b}}
\newcommand{\UMeas}{\mathfrak{M}_\mathsf{u}}
\newcommand{\CUMeas}{\mathfrak{M}_{\mathsf{u}\sigma}}
\newcommand{\cspace}{\Call([0,\omega_1])}
\newcommand{\sspace}{\mathsf{M}}
\newcommand{\precom}{\mathsf{p}}
\newcommand{\G}{\mathcal{G}}
\newcommand{\N}{\mathcal{N}}
\newcommand{\F}{\mathcal{F}}
\newcommand{\Pcov}{\mathcal{P}}
\newcommand{\U}{\mathcal{U}}
\newcommand{\W}{\mathcal{W}}
\newcommand{\ball}{\circledcirc}
\newcommand{\modik}[1]{\mathord{\stackrel{#1}{\vee}}}
\newcommand{\vbeta}{B}
\newcommand{\vbetaj}{B_j}
\newcommand{\vbetap}{B^\prime}
\newcommand{\charf}[1]{\mathrm{I}_{#1}}

\title{Cauchy filters from Pelant's games}

\author{Jan Pachl  \\
        Fields Institute   \\
        Toronto, Ontario, Canada}
\date{October 7, 2013}

\begin{document}
\maketitle

\begin{abstract}
The language of finite games is used to rephrase Pelant's proof
of his result:
The separable modification of the complete metric space
$\cspace$ is not complete.
\end{abstract}

\section{Introduction}

The uniform space concepts not defined here may be found in
Isbell~\cite{Isbell1964us}.

For every uniform space $X$ there is another uniform space on the same
set of points and compatible with the same topology,
called here the \emph{separable modification of $X$} and
denoted by $\precom_1 X$.
The uniformity of $\precom_1 X$ is projectively generated
by all uniformly continuous mappings from $X$ to separable metric spaces.
Isbell's notation for $\precom_1 X$ is $X_{\aleph_1}$
and also $eX$
(pages 52 and 129 in~\cite{Isbell1964us}, respectively).

What follows is a result of my attempt to understand
Pelant's proof~\cite{Pelant1975rnp}\cite{Pelant1976pcu}
of his theorem:
\begin{center}
\textbf{There is a (not too large) complete metric space $X$ for which
$\precom_1 X$ is not complete.}
\end{center}
The ``not too large'' qualification excludes examples such as the discrete
space of measurable cardinality.
In fact, in the spirit of~\cite[Thm 1.1]{Pelant2006csw},
Pelant's method lets us take $X=\cspace$,
the Banach space of continuous real-valued functions on the compact space
$[0,\omega_1]$ with the sup norm $\norm{\cdot}$.
In this paper I describe Pelant's construction
using level sets of functions in $\cspace$ and finite games
instead of Pelant's cornets and finite sequences
of alternating quantifiers.

By virtue of Exercise 2(b) on page 52 in~\cite{Isbell1964us},
every countable uniform cover of $X$ is a uniform cover of $\precom_1 X$.
Hence the set of all countable uniform covers of $X$
is a basis of uniform covers for $\precom_1 X$.
However, for the construction in the next section the property
of being a \emph{point-finite} uniform cover turns out to be
more useful than being countable.
The incompleteness of $\precom_1 \cspace$
then follows by using a basis of point-finite covers.
This is explained in section~\ref{sec:corollaries},
along with several other consequences of the main theorem.

\mbox{}\\
\textbf{Acknowledgments}\\
This paper was written while I was a visitor at the Fields Institute
in Toronto.
I wish to thank Juris Stepr\={a}ns for his comments that helped me
improve and simplify the presentation.

\section{The construction}
    \label{sec:main}

As in the proof of Theorem~1.1 in~\cite{Pelant2006csw}, let
\[
\sspace:=\{f\iin\cspace \mid f \text{ is monotone non-increasing, }
f(0)=1 \text{ and } f(\omega_1)=0 \}.
\]
With the subspace metric defined by the sup norm $\norm{\cdot}$ on $\cspace$,
$\sspace$ is a complete metric space.

\begin{maintheorem}
Let $\U$ be a uniform structure on $\sspace$ such that
\begin{itemize}
\item
the $\norm{\cdot}$ topology on $\sspace$ is compatible with $\U$;
\item
the $\norm{\cdot}$ uniformity on $\sspace$ is finer than $\U$; and
\item
$\U$ has a uniformity basis consisting of point-finite covers.
\end{itemize}
Then $\U$ is not complete.
\end{maintheorem}

The proof of the main theorem in this section is based
on the proof of Th.~17 in~\cite{Pelant1976pcu} (pp. 58--60).
It also incorporates elements of the proof of Th.~1.1
in~\cite{Pelant2006csw}.
It should be noted that \cite[Th.17]{Pelant1976pcu} deals
with more general point characters of uniformities;
the point-finite version that I prove here is a special case.

Write $\posint:=\{1,2,\dotsc\}$.
For every ordinal $\beta<\omega_1$ let
$\charf{\beta}\colon [0,\omega_1]\to \{0,1\}$ be
the characteristic function of the closed interval $[0,\beta]$.
When $n\iin\posint$
and $\vbeta=\langle \beta_0,\beta_1,\dotsc,\beta_{n-1}\rangle$
is a finite sequence of ordinals $<\omega_1$,
let
\[
\charf{\vbeta} := \max_{0\leq k < n} \frac{(n-k)\,\charf{\beta_k}}{n}
\]
and note that $\charf{B}\iin\sspace$.
When $g\iin\sspace$ and $\varepsilon>0$, write
\[
\ball[g,\varepsilon]:= \{f\iin\sspace \mid \;\norm{f-g}\leq\varepsilon \}.
\]
Covers of the form
$\{ \ball[g,\varepsilon] \mid g\iin\sspace \}$, $\varepsilon>0$,
form a uniformity basis of the metric space $\sspace$.
Hence any uniformity $\U$ satisfying the assumptions of the main theorem
has a uniformity basis consisting of point-finite covers each of which
is refined by the cover
$\{ \ball[f,\varepsilon] \mid f\iin\sspace \}$ for some $\varepsilon>0$.

For $U\subseteq\sspace$, $f\iin\sspace$ and $n\iin\posint$,
the finite game $\G(U,f,n)$ is played by two players Alice and Bob
as follows:
Alice moves first, and then the players alternate in choosing countable
ordinals;
each choice must be larger than or equal to the ordinals already
chosen in previous moves.
The game ends when the players have made $n$ moves each.

Thus each run of the game is a sequence
$
\langle \alpha_0 , \beta_0 , \alpha_1 , \beta_1 , \dotsc ,
\beta_{n\!-2}, \alpha_{n\!-1} , \beta_{n\!-1} \rangle
$
of ordinals such that
$
\alpha_0 \leq \beta_0 \leq \alpha_1 \leq \beta_1 \leq \dotsc \leq
\beta_{n\!-2}\leq \alpha_{n\!-1} \leq \beta_{n\!-1} < \omega_1
$;
say that \emph{Bob wins this run}
iff $\ball[f\vee\charf{\vbeta},1/n]\subseteq U$,
where $\vbeta$ is the sequence
$\langle \beta_0,\beta_1,\dotsc,\beta_{n\!-1}\rangle$ of Bob's moves.

A \emph{game position} is a prefix of a run of the game.
A \emph{strategy} for Bob is a mapping that takes any $U$, $f$, $n$ and
a game position ending with Alice's move as inputs
and produces a countable ordinal as output, to be used as Bob's next move.
Say that \emph{Bob wins the game} $\G(U,f,n)$ if he has a winning strategy.

We also need the modified game $\G(U,f,n; k)$ for every $1\leq k\leq n$;
it has the same rules as $\G(U,f,n)$ and in addition the first $n-k$ moves
by each player must be zero;
that is, $\alpha_i=\beta_i=0$ for $0\leq i\leq n-k-1$.
Thus $\G(U,f,n; n)$ is $\G(U,f,n)$.

Another modified game is $\G(U,f,n; k , \alpha)$ for $1\leq k\leq n$ and
$\alpha<\omega_1$.
It has the same rules as $\G(U,f,n; k)$ and in addition Alice
must play $\alpha_{n-k}=\alpha$.

For $f\iin\sspace$, $n\iin\posint$, $1\leq k\leq n$ and $\alpha<\omega_1$,
write
\begin{align*}
\W(f,n) & := \{ U \subseteq \sspace \mid \text{ Bob wins } \G(U,f,n) \}  \\
\W(f,n; k) & := \{ U \subseteq \sspace
                   \mid \text{ Bob wins } \G(U,f,n; k) \} \\
\W(f,n; k , \alpha ) & := \{ U \subseteq \sspace
                   \mid \text{ Bob wins } \G(U,f,n; k , \alpha) \}
\end{align*}

In the proof of the main theorem, the following lemma is used
to show that a certain $\U$-Cauchy filter of subsets of $\sspace$
does not converge.

\begin{lemma}
    \label{lem:largediam}
If $U\iin\W(f,n)$ for some $f\iin\sspace$ and some $n\iin\posint$
then the $\norm{\cdot}$ diameter of $U$ is $1$.
\end{lemma}
\begin{proof}
Play the game $\G(U,f,n)$ twice, both times with Bob using
his winning strategy.
In the first run Alice plays any legal moves.
The run produces $f\vee\charf{\vbeta}\iin U$,
where $\vbeta$ is the sequence of Bob's moves.
Since $f\vee\charf{\vbeta}\iin\cspace$
and $f\vee\charf{\vbeta}(\omega_1)=0$,
there is $\alpha<\omega_1$ such that
$f\vee\charf{\vbeta}(\alpha)=0$.
Play the game $\G(U,f,n)$ again;
this time Alice's first move is $\alpha$ and she plays any legal moves
after that.
Let $\vbetap$ be the sequence of Bob's moves in the second run.
Then $f\vee\charf{\vbetap}\iin U$ and $f\vee\charf{\vbetap}(\alpha)=1$.
\end{proof}

By the next lemma, every $\W(f,n; k)$ generates
a $\sigma$-filter of subsets of $\sspace$.

\begin{lemma}
    \label{lemma:intersect}
Let $f\iin\sspace$, $n\iin\posint$, $1\leq k\leq n$,
and $U_j\iin \W(f,n; k)$ for $j\iin\omega$.
Then
\[
\bigcap_{j\in\omega} U_j \in \W(f,2n; 2k).
\]
\end{lemma}

\begin{proof}
Write $U:=\bigcap_{j\in\omega} U_j$.
For every $j\iin\omega$ Bob has a winning strategy $S_j$
for the game $\G(U_j,f,n; k)$.
Bob's winning strategy for $\G(U,f,2n; 2k)$ is the following:
In the game position
$\alpha_0 \leq \beta_0 \leq \dotsc \leq \beta_{i-1}\leq \alpha_{i}$
where $i\leq 2n-1$ is even,
Bob's next move is $\beta_i := \sup_{j\in\omega} \beta_{j,i}$,
where $\beta_{j,i}$ is chosen by Bob's strategy $S_j$
in the position
\[
\alpha_0 \leq \beta_0 \leq \alpha_2 \leq \beta_2 \leq\dotsc
\leq \alpha_{i-2}\leq \beta_{i-2}\leq \alpha_{i}
\]
of the game $\G(U_j,f,n; k)$.
In the game position
$\alpha_0 \leq \beta_0 \leq \dotsc \leq \beta_{i-1}\leq \alpha_{i}$
where $i\leq 2n-1$ is odd, Bob's next move is $\beta_i:=\alpha_i$.

Let $\vbeta =\langle \beta_0,\beta_1,\dotsc,\beta_{2n\!-1}\rangle$
be the resulting sequence of Bob's moves.
For every $j\iin\omega$ let $\vbetaj$ be the sequence
$\langle \beta_{j,0},\beta_{j,2},\dotsc,\beta_{j,2n\!-2}\rangle$
of the choices made by strategy $S_j$
in the positions
$\alpha_0 \leq \beta_0 \leq \alpha_2 \leq \beta_2 \leq\dotsc
\leq \alpha_{i-2}\leq \beta_{i-2}\leq \alpha_{i}$ with $i$ even.
Then $\ball[f\vee\charf{\vbetaj},1/n]\subseteq U_j$ for every $j$
because $S_j$ is a winning strategy for Bob in $\G(f,U_j,n; k)$.
But $\norm{(f\vee\charf{\vbeta}) - (f\vee\charf{\vbetaj})} \leq 1/2n$
for every $j$ and therefore
$\ball[f\vee\charf{\vbeta},1/2n]\subseteq U$.
\end{proof}

\begin{corollary}
    \label{cor:double}
Let $f\iin\sspace$ and $n\iin\posint$.
Then $\W(f,n)\subseteq \W(f,2n)$.
\qed
\end{corollary}

\begin{corollary}
    \label{cor:finite}
Let $\Pcov$ be a point-finite cover of
$\sspace$, $f\iin\sspace$, $n\iin\posint$
and $1\leq k\leq n$.
Then the set $\Pcov \cap \W(f,n; k)$ is finite.
\end{corollary}

\begin{proof}
Assume to the contrary that there is a countable infinite set
$\N\subseteq\Pcov\cap \W(f,n; k)$.
Then $\bigcap\N \in \W(f,2n; 2k)$ by Lemma~\ref{lemma:intersect}.
Hence $\bigcap\N\neq\emptyset$,
which contradicts $\Pcov$ being point-finite.
\end{proof}

Our next goal is to prove that the finite set $\Pcov \cap \W(f,n; k)$
is not empty.
This is the crucial step in the proof of the main theorem.

\begin{lemma}
    \label{lemma:mainstep}
Let $\Pcov$ be a point-finite cover of $\sspace$, and let $n\iin\posint$
be such that $\Pcov$ is refined by the cover
$\{ \ball[g,2/n] \mid g\iin\sspace \}$.
Then $\Pcov \cap \W(f,n; k) \neq \emptyset$ for every $f\iin\sspace$
and every $1\leq k\leq n$.
\end{lemma}

\begin{proof}
In this proof letters $\alpha$ and $\beta$, with or without subscripts,
stand for countable ordinals.
For $f\iin\sspace$, $1\leq k\leq n$ and $\alpha$,
define the function $f\modik{k}\alpha \iin \sspace$ by
$f\modik{k}\alpha := f \vee (k \,\charf{\alpha} / n)$.

The proof proceeds by induction on $k$, starting with $k=1$.
In the game $\G(f,n; 1)$ all except the last move by each player are 0.
Hence
\[
\W(f,n; 1) = \{U\subseteq\sspace
                \mid \forall \alpha \;\; \exists \beta \geq \alpha \; : \;
                \ball[f\modik{1} \beta,1/n] \subseteq U \}.
\]
Since $\Pcov$ is refined by $\{ \ball[g,2/n] \mid g\iin\sspace \}$,
there is $U_0\iin\Pcov$ for which $\ball[f,2/n]\subseteq U_0$.
Since $\ball[f\modik{1} \beta,1/n] \subseteq \ball[f,2/n]$ for every $\beta$,
it follows that Bob wins $\G(U_0,f,n;1)$ with \emph{any} strategy.
Hence $U_0\iin \Pcov \cap \W(f,n;1)$.
That concludes the proof for $k=1$.

For the induction step, take $k\leq n-1$ and
assume the conclusion of the lemma holds for every $f\iin\sspace$.
Take any $f\iin\sspace$.

\emph{Claim:}
For every $\alpha$ there exists $U\iin\Pcov$ for which Bob wins
$\G(U,f,n; k+1 , \alpha)$.

By the induction assumption with $f\modik{k+1}\alpha$ in place of $f$ we have
$\Pcov \cap \W(f\modik{k+1} \alpha,n; k) \neq \emptyset$,
hence there are $U\iin\Pcov$ and a winning strategy $S$ for Bob in the game
$\G(U,f\modik{k+1} \alpha,n; k)$.
Bob's winning strategy for $\G(U,f,n; k+1,\alpha)$ is to choose
$\beta_{n-k-1}=\alpha$ and then follow strategy $S$ in the subsequent moves.
That proves the claim.

Now observe that for $\alpha \leq \alpha^\prime$ we have
$\W(f,n; k+1, \alpha) \supseteq \W(f,n; k+1 , \alpha^\prime)$.
Since the sets $\Pcov \cap \W(f,n; k+1,\alpha)$ are finite by
Corollary~\ref{cor:finite},
there is $\alpha_0$ such that
\[
\Pcov \cap \W(f,n; k+1,\alpha_0)
\subseteq \Pcov \cap \W(f,n; k+1, \alpha)
\]
for every $\alpha$.
We have $\Pcov \cap \W(f,n; k+1, \alpha_0)\neq\emptyset$ by the claim,
and Bob wins $\G(U,f,n; k+1)$ for every
$U\iin \Pcov \cap \W(f,n; k+1, \alpha_0)$.
That completes the induction step.
\end{proof}

\begin{lemma}
    \label{lemma:filter}
Let $\U$ be a uniform structure on the set $\sspace$ such that
the metric uniformity of $\sspace$ is finer than $\U$ and
$\U$ has a uniformity basis consisting of point-finite covers.
Let $f\iin\sspace$.
Then there is a filter $\F$ of subsets of $\sspace$ such that
\begin{itemize}
\item
$\F$ is $\U$-Cauchy; and
\item
for every $U\iin\F$ there is $n\iin\omega$ for which Bob wins $\G(U,f,2^n)$.
\end{itemize}
\end{lemma}

\begin{proof}
For every $\U$-uniform cover $\Pcov$
there is $r(\Pcov)\iin\omega$ such that $\Pcov$ is refined by
$\{ \ball[g,2^{1\!-n}] \mid g\iin\sspace \}$ for every $n\geq r(\Pcov)$.
Define
\[
\F_0 := \left\{\; \bigcap \left( \Pcov \cap \W(f,2^n) \right)
\mid \Pcov \text{ is a point-finite }
\U\text{-uniform cover and } n\geq r(\Pcov) \right\} .
\]
If $\F_0$ has the finite intersection property then clearly the filter $\F$
generated by $\F_0$ is $\U$-Cauchy.

Take any finite subset $\{A_0, A_1, \dotsc, A_j\}$ of $\F_0$
and let $A:=\bigcap_{i=0}^j A_i$.
There are point-finite $\U$-uniform covers $\Pcov_i$ and $n_i$
such that $n_i \geq r(\Pcov_i)$
and $A_i = \bigcap \left( \Pcov_i \cap \W(f,2^{n_i}) \right)$
for $0\leq i \leq j$.
Write $n:= \max_i n_i$.
The sets $\Pcov_i \cap \W(f,2^{n_i})$ are finite
by Corollary~\ref{cor:finite},
hence
\[
A = \bigcap_{0\leq i\leq j} \;
     \bigcap \left( \Pcov_i \cap \W(f,2^{n_i}) \right)
\in \W(f,2^{n+1})
\]
by Lemma~\ref{lemma:intersect} and Corollary~\ref{cor:double}.
Thus $\F_0$ has the finite intersection property and the filter generated
by $\F_0$ has the second property in the lemma.
\end{proof}

\begin{proof}[Proof of the main theorem]
Take any $f\iin\sspace$ (for example $f=\charf{0}$).
Let $\F$ be a filter with the two properties in Lemma~\ref{lemma:filter}.
By Lemma~\ref{lem:largediam}, $\F$ does not converge in the topology of
$\sspace$.
\end{proof}


\section{Corollaries}
    \label{sec:corollaries}

Following Pelant~\cite{Pelant1976pcu},
I have written the main theorem in a form that not only applies
to the space $\precom_1\sspace$,
but is also useful for questions about point-finite covers.

\subsection{Completeness of the separable modification}
    \label{subs:Compl}

All that is now needed to prove that the space $\precom_1\cspace$
is not complete
is the following result of Vidossich~\cite{Vidossich1970usc}.

\begin{lemma}
    \label{lemma:vidossich}
Let $X$ be any uniform space.
The uniformity $\precom_1 X$ has a basis consisting
of countable point-finite covers.
\qed
\end{lemma}

Combining the main theorem in section~\ref{sec:main}
with Lemma~\ref{lemma:vidossich}, we obtain:

\begin{theorem}
    \label{th:notcompl}
The uniform spaces $\precom_1\sspace$ and  $\precom_1\cspace$
are not complete.
\qed
\end{theorem}

Pelant~\cite{Pelant1975rnp} constructed a complete metric space $X$
for which $\precom_1 X$ is not complete.
Although using a different terminology,
the proof of Theorem~\ref{th:notcompl} here is a modification of
Pelant's construction in~\cite{Pelant1975rnp} and~\cite{Pelant1976pcu}.
The fact that we can take $X=\cspace$ is not surprising
in view of~\cite[Thm 1.1]{Pelant2006csw}.

The theorem has an interesting application in the theory of
\emph{uniform measures}~\cite{Pachl2013usm}.
For any uniform space $X$, let $\Ub(X)$ be the space of bounded
real-valued functions on $X$.
Let $\UMeas(X)$ be the space of uniform measures;
that is, the linear functionals on $\Ub(X)$ that are continuous on every
bounded uniformly equicontinuous subset of $\Ub(X)$
in the $X$-pointwise topology.

In their work on topological centres, Ferri and Neufang~\cite{Ferri2007tca}
defined also the space of linear functionals $\Ub(X)$
that are \emph{sequentially}
continuous on bounded uniformly equicontinuous subsets of $\Ub(X)$.
In section~8.1 of~\cite{Pachl2013usm},
where $\CUMeas(X)$ denotes the space of such functionals,
I prove that $\CUMeas(X)=\UMeas(\precom_1 X)$.
The following corollary is then an immediate consequence of
Theorem~\ref{th:notcompl}, using the relationship between the space
$\UMeas(X)$ and the completion of $X$ (section 6.5 in~\cite{Pachl2013usm}).

\begin{corollary}
    \label{cor:cumeas}
$\CUMeas(\sspace)\neq\UMeas(\sspace)$
and $\CUMeas(\cspace)\neq\UMeas(\cspace)$.
\qed
\end{corollary}

I don't know any simple proof of Corollary~\ref{cor:cumeas},
or in fact any other construction, without using measurable cardinals,
of a uniform space $X$ such that $\CUMeas(X)\neq\UMeas(X)$.

\subsection{Point-finite refinements of uniform covers}
    \label{subs:PFinite}

Let $X$ be a metric space.
Since $X$ is paracompact, every open cover of $X$ is refined by an open
locally finite cover.
Thus it is natural to ask, as Stone~\cite{Stone1960uss} and then
Isbell~\cite[p.144]{Isbell1964us} did:
Is every \emph{uniform} cover of $X$ refined by a uniformly
locally finite \emph{uniform} cover?
By~\cite[VIII.3]{Isbell1964us}, an equivalent question is:
\textbf{Is every uniform cover of $X$ refined by a point-finite
uniform cover?}

Pelant~\cite{Pelant1975crp}\cite{Pelant1977cpu}
and {\v{S}}{\v{c}}epin~\cite{Scepin1975opi} constructed metric spaces $X$
for which the answer is negative.
In a later paper, Pelant, Holick{\'y} and Kalenda~\cite{Pelant2006csw}
prove that the answer is no for $X=\cspace$;
this immediately follows also from the main theorem
in section~\ref{sec:main}.

Let $X$ be a Banach space with
the metrizable uniformity defined by its norm.
Pelant~\cite{Pelant1994eic} proved that $X$ has a uniformity basis
consisting of point-finite covers if and only if $X$ is uniformly
homeomorphic to a subset of $c_0(\Gamma)$ for some index set $\Gamma$.
Thus we get Theorem~1.1 in~\cite{Pelant2006csw}:
The Banach space $\cspace$ is not uniformly homeomorphic to a subset
of $c_0(\Gamma)$ for any $\Gamma$.
This and related results are discussed
in more detail in~\cite{Pelant2006csw}.

\end{document}